\documentclass{amsart}
\usepackage{amsmath}
\usepackage{amssymb}
\usepackage{amsthm}
\usepackage{enumerate}

\newtheorem{theorem}{Theorem}
\newtheorem{prop}[theorem]{Proposition}
\newtheorem{lemma}[theorem]{Lemma}
\newtheorem{rem}[theorem]{Remark}
\newtheorem{exmp}[theorem]{Example}

\begin{document}

\title[Recovering of the Grassmann graph]{Recovering of the Grassmann graph from the subgraph of non-degenerate subspaces}
\author{Mark Pankov}
\address{Faculty of Mathematics and Computer Science, 
University of Warmia and Mazury, S{\l}oneczna 54, 10-710 Olsztyn, Poland}
\email{pankov@matman.uwm.edu.pl}

\begin{abstract}
Let ${\mathbb F}$ be a (not necessarily finite) field. 
A subspace of the vector space ${\mathbb F}^n$ is called {\it non-degenerate} if it is not contained in a coordinate hyperplane. 
We show that the Grassmann graph of $k$-dimensional subspaces of ${\mathbb F}^n$, $1<k<n-1$,
can be recovered from the subgraph of non-degenerate subspaces if $|{\mathbb F}|>n-k$.
In the case when ${\mathbb F}={\mathbb F}_q$ is the field of $q$ elements, this subgraph is known as 
the graph of non-degenerate linear $[n,k]_q$ codes.
\end{abstract}

\maketitle

\section{Introduction}
Linear $[n,k]_q$ codes are $k$-dimensional subspaces of the vector space ${\mathbb F}^n_q$, where ${\mathbb F}_q$ is the field of $q$ elements,
see \cite[Section 1.1]{TVN}.
In practice, such a code may be useful only if it is non-degenerate, i.e. it is not contained in a coordinate hyper\-plane,  
since degenerate codes actually are of length less than $n$.
The graph of non-degenerate linear $[n,k]_q$ codes $\Gamma(n,k)_q$ is the subgraph of the Grassmann graph of $k$-dimensional subspaces of ${\mathbb F}^n_q$
induced by the set of non-degenerate subspaces;
two distinct codes are adjacent vertices of $\Gamma(n,k)_q$ if and only if they have the maximal number of common code words.
This graph is first considered in \cite{KP1} and the research is continued in  \cite{KP2,P1,P2}. 
Structural properties of $\Gamma(n,k)_q$ depend significantly on the parameters $n,k,q$ (see, for example, \cite{KP1,P1}).
Also, in contrast to Grassmann graphs, there is no duality: the Grassmann graphs of $k$-dimensional and $(n-k)$-dimensional  subspaces of ${\mathbb F}^n_q$ are isomorphic, 
for the graphs of non-degenerate linear codes this fails. 

Subgraphs of Grassmann graphs induced by various types of linear codes are investigated in 
\cite{CGK,CG, KPP, KP3,KP4,PPZ}.

In the present paper, we show that the Grassmann graph of $k$-dimensional subspaces of ${\mathbb F}^n$, where 
${\mathbb F}$ is a not necessarily finite field and $1<k<n-1$,  can be recovered from the subgraph of non-degenerate subspaces if 
$|{\mathbb F}|>n-k$. Our resoning is in spirit of recovering of projective spaces from affine spaces.
Recall that an affine space can be obtained from a projective space by removing a hyperplane.
To recover the projective space it is sufficient to add to the affine space the classes of parallel lines as points at infinity,
such points form the removed hyperplane. 
To recover the Grassmann graph we determine some special classes of disjoint maximal cliques in the subgraph of non-degenerate subspaces 
and show that the extensions of cliques from such a class to maximal cliques of the Grassmann graph are intersecting in a vertex which is a degenerate subspace.
So, the Grassmann graph can be obtained from the subgraph of non-degenerate subspaces by 
adding these classes of cliques as vertices. 
There are some technical differences for $n=4$ and we consider this case separately.

Examples from Section 7 show that the provided construction is impossible if $|{\mathbb F}|\le n-k$. 
We conjecture that recovering of the Grassmann graph in this case can be obtained in a more complicated way.

\section{Grassmann graph and the subgraph of non-degenerate subspaces}

Let ${\mathbb F}$ be a field (not necessarily finite). Consider the $n$-dimensional vector space $V={\mathbb F}^n$. 
The kernel of the $i$-th coordinate functional $(x_1,\dots,x_n)\to x_i$
will be denoted by $C_i$. A subspace of $V$ is said to be {\it degenerate} if there is $C_i$ containing it,
in other words, the restriction of a certain coordinate functional to this subspace is zero; otherwise, it  will be called {\it non-degenerate}. 

Let $k\in \{1,\dots,n-1\}$.
The {\it Grassmann graph} $\Gamma_k(V)$ is the simple graph whose vertices are $k$-dimensional subspaces of $V$
and two such subspaces are adjacent vertices of  this graph (vertices connected by an edge) if and only if their intersection is $(k-1)$-dimensional. 
For this reason, we say that two $k$-dimensional subspaces are {\it adjacent} if their intersection is $(k-1)$-dimensional which is equivalent to 
the fact that the sum of these subspaces is $(k+1)$-dimensional.

Denote by $\Delta_k(V)$ the subgraph of $\Gamma_k(V)$ induced by the set of non-degenerate $k$-dimensional subspaces. 
In other words, the vertex set of $\Delta_k(V)$ is the set of non-degenerate $k$-dimensional subspaces and 
two such subspaces are adjacent vertices of $\Delta_k(V)$ if and only if they are adjacent vertices of $\Gamma_k(V)$.
If ${\mathbb F}={\mathbb F}_q$ is the field of $q$ elements, then $k$-dimensional subspaces of $V$ are identified with 
linear $[n,k]_q$ codes (see \cite[Section 1.1]{TVN}) and $\Delta_k(V)$ is the graph of non-degenerate linear $[n,k]_q$ codes $\Gamma(n,k)_q$ investigated in \cite{KP1,KP2,P1,P2}. 

For $k=1,n-1$ the graphs $\Gamma_k(V)$ and $\Delta_k(V)$ are complete (any two distinct vertices are adjacent). 
In what follows, we will always assume that $1<k<n-1$ which implies that $n\ge 4$. 

\begin{theorem}\label{theorem-1}
If $|{\mathbb F}|>n-k$, then the Grassmann graph $\Gamma_k(V)$ can be recovered from the subgraph of non-degenerate subspaces $\Delta_k(V)$.
\end{theorem}

\section{Maximal cliques}
\subsection{Maximal cliques of the Grassmann graph}
It is well-known that there are precisely the following two types of maximal cliques of $\Gamma_k(V)$:
\begin{enumerate}
\item[$\bullet$] the star ${\mathcal S}(X)$ consisting of all $k$-dimensional subspaces containing a certain $(k-1)$-dimensional subspace $X$;
\item[$\bullet$] the top ${\mathcal T}(Y)$ formed by all  $k$-dimensional subspaces contained in a certain $(k+1)$-dimensional subspace $Y$. 
\end{enumerate}
The intersection of two distinct stars or two distinct tops is empty or a single vertex. 
Two distinct stars have a non-empty intersection if and only if the corresponding $(k-1)$-dimensional subspaces are adjacent.
Similarly, the intersection of two distinct tops is non-empty if and only if the corresponding $(k+1)$-dimensional subspaces are adjacent.

The intersection of a star and a top  is empty or it contains more than one vertex
(the second possibility is realized if and only if  the corresponding $(k-1)$-dimensional and $(k+1)$-dimensional subspaces are incident).
In the second case, the intersection is called a {\it line} of $\Gamma_k(V)$; 
it contains precisely $q+1$ vertices if ${\mathbb F}={\mathbb F}_q$ and infinitely many if ${\mathbb F}$ is infinite.
Stars and tops together with lines contained in them are projective spaces of (projective) dimension $n-k$ and $k$, respectively. 
If ${\mathbb F}={\mathbb F}_q$, then  stars and tops contain $[n-k+1]_q$ and $[k+1]_q$ vertices (respectively), where 
$$[t]_q=\frac{q^t-1}{q-1}=q^{t-1}+\dots+q+1$$
is the number of points in ${\rm PG}(t-1,q)$.

\subsection{Maximal cliques in the subgraph of non-degenerate subspaces}
Following \cite{KP2},
for every $(k-1)$-dimensional subspace $X$ and every $(k+1)$-dimensional subspace $Y$
we denote by ${\mathcal S}^c(X)$ and ${\mathcal T}^c(Y)$ the intersections of the star ${\mathcal S}(X)$ and the top ${\mathcal T}(X)$
with the set of non-degenerate subspaces. 
Each of these intersections is a (not necessarily maximal) cliques of $\Delta_k(V)$ (if it is not empty). 
For example, if $Y$ is degenerate, then ${\mathcal T}^c(Y)$ is empty. 
We say that ${\mathcal S}^c(X)$ or ${\mathcal T}^c(Y)$ is  a {\it star} or, respectively, a {\it top} of $\Delta_k(V)$
only when it is a maximal clique of  $\Delta_k(V)$. 
If $X$ is non-degenerate, then every subspace containing $X$ is non-degenerate and ${\mathcal S}(X)={\mathcal S}^c(X)$
is a maximal clique in both $\Gamma_k(V)$ and $\Delta_k(V)$; such cliques are said to be {\it maximal stars} of $\Delta_k(V)$.
Since each clique of $\Delta_k(V)$ is a clique of $\Gamma_k(V)$, 
every maximal clique of $\Delta_k(V)$  is a star or a top.

\begin{prop}\label{prop-star}
If $|{\mathbb F}|\ge 3$, then ${\mathcal S}^c(X)$ is a star of $\Delta_k(V)$ for every $(k-1)$-dimensional subspace $X$
and  there is no $(k+1)$-dimensional subspace $Y$ such that ${\mathcal S}^c(X)={\mathcal T}^c(Y)$.
\end{prop}

\begin{proof}
Observe that for every proper subspace $Y\subset V$ there is a non-degenerate $1$-dimensional subspace 
which is not contained in $Y$. This is clear if $Y$ does not contain non-degenerate $1$-dimensional subspaces.
Suppose that $x=(x_1,\dots, x_n)\in Y$ and every $x_i$ is non-zero. 
Since $Y$ is defined by a system of linear equations and $|{\mathbb F}|\ge 3$, for a certain $i$ there is non-zero $x'_i$ distinct from $x_i$ and such that
the vector obtained from $x$ via replacing $x_i$ by $x'_i$ does not belong to $Y$.

For every $(k+1)$-dimensional subspace $Y$ containing $X$
we take a non-degenerate $1$-dimensional subspace $P$ which is not contained in $Y$. 
Then $X+P$ belongs to ${\mathcal S}^c(X)$ and does not belong to ${\mathcal T}^c(Y)$. 
So, there is no ${\mathcal T}^c(Y)$ containing ${\mathcal S}^c(X)$. 
In particular, ${\mathcal S}^c(X)$ contains more than one vertex 
(indeed, if ${\mathcal S}^c(X)=\{A\}$, then ${\mathcal S}^c(X)\subset {\mathcal T}^c(Y)$ for every $(k+1)$-dimensional subspace 
$Y$ containing $A$). 
Since the intersection of two distinct stars of $\Gamma_k(V)$ contains at most one vertex,
${\mathcal S}^c(X)$ is a maximal clique of $\Delta_k(V)$.
\end{proof}

\begin{rem}{\rm 
If $|{\mathbb F}|=2$, then ${\mathcal S}^c(X)$ is a star of $\Delta_k(V)$ only when
$X$ is contained in at most $n-k-1$ distinct $C_i$,
see \cite[Proposition 2]{KP2}.
}\end{rem}

Let $Y$ be a non-degenerate $(k+1)$-dimensional subspace of $V$. 
Then ${\mathcal T}^c(Y)$ is obtained from ${\mathcal T}(Y)$ by removing 
$$Y\cap C_1,\dots, Y \cap C_n$$
(it is possible that $Y\cap C_i=Y\cap C_j$ for some distinct $i,j$).
By  \cite[Lemma 2]{KP2}, we have
$$\max\{0, [k + 1]_q - n\} \le |{\mathcal T}^c(Y)| \le [k + 1]_q -k-1.$$

\begin{prop}\label{prop-top}
The following assertions are fulfilled:
\begin{enumerate}
\item[{\rm (1)}] If ${\mathbb F}={\mathbb F}_q$,  then ${\mathcal T}^c(Y)$ is a top of $\Delta_k(V)$ for every 
non-degenerate $(k+1)$-dimensional subspace $Y$ only when 
$$[k + 1]_q-(q+1)> n.$$
\item[{\rm (2)}] If ${\mathbb F}$ is infinite, then ${\mathcal T}^c(Y)$ is a top of $\Delta_k(V)$ for every 
non-degenerate $(k+1)$-dimensional subspace $Y$.
\end{enumerate}
\end{prop}

\begin{proof}
The first statement is known, see \cite[Corollary 1]{KP2}. 

If ${\mathbb F}$ is infinite, then ${\mathcal T}^c(Y)$ contains infinitely many vertices 
and there are at most $n$ distinct $k$-dimensional subspaces of $Y$ which are degenerate and, consequently,  do not belong to ${\mathcal T}^c(Y)$. 
For every $(k-1)$-dimensional subspace $X\subset Y$ there are infinitely many $k$-dimensional subspaces of 
$Y$ which do not contain $X$; almost all of them (except some finite number) are non-degenerate
which means that ${\mathcal T}^c(Y)$ is not contained  ${\mathcal S}^c(X)$.
So, there is no ${\mathcal S}^c(X)$ containing ${\mathcal T}^c(Y)$; in particular, ${\mathcal T}^c(Y)$ contains more than one vertex. 
Since the intersection of two distinct tops of $\Gamma_k(V)$ contains at most one vertex,
${\mathcal T}^c(Y)$ is a maximal clique of $\Delta_k(V)$.
\end{proof}

Note that a star is not a top and vice versa
(see Proposition \ref{prop-star} for the case $|{\mathbb F}|\ge 3$ and \cite[Proposition 5]{P1} for the general case). 
This implies that every maximal clique of $\Delta_k(V)$ (a star or a top) contains at least three vertices.

As in the Grassmann graph, the intersection of two distinct stars or two distinct tops of $\Delta_k(V)$ contains at most one vertex.
The intersection of two distinct maximal cliques of $\Delta_k(V)$ will be called a {\it line} of $\Delta_k(V)$ if it contains more than one vertex. 
This is possible only when maximal cliques are a star and a top. 
The following example shows that there are pairs of stars and tops intersecting precisely in one vertex. 

\begin{exmp}\label{exmp1}{\rm
Suppose that ${\mathbb F}={\mathbb F}_q$ and $q=n-k+1\ge 3$. 
Consider the star ${\mathcal S}^c(X)$, where 
$$X=C_1\cap\dots\cap C_{n-k+1}$$
(this is a star of $\Delta_k(V)$ by Proposition \ref{prop-star}).
Let $P$ be the $1$-dimensional subspace containing the vector $(1,\dots,1)$
and let $Q$ be any $1$-dimensional subspace containing a vector whose first $n-k+1=q$ coordinates are mutually distinct elements of ${\mathbb F}_q$. 
The  $(k+1)$-dimensional subspace $$Y=X+P+Q$$ is non-degenerate and ${\mathcal T}^c(Y)$ is a top of $\Delta_k(V)$
(it will be shown later that  ${\mathcal T}^c(Y)$ is a top of $\Delta_k(V)$ for every non-degenerate $(k+1)$-dimensional subspace $Y$ if $|{\mathbb F}|>n-k$, see 
Lemma \ref{lemma-top}).
If $t_1,\dots, t_{n-k+1}$ are the elements of ${\mathbb F}_q$,
then  every $1$-dimensional subspace of $P+Q$ distinct from $P$ contains the vector 
$$(t_i,\dots,t_i)-(t_1,\dots,t_{n-k+1},\dots)$$
for a certain $i\in \{1,\dots,n-k+1\}$.
This means that the hyperplanes $C_1,\dots, C_{n-k+1}$ intersect $P+Q$ in mutually distinct $1$-dimensional subspaces.
Every $k$-dimensional subspace belonging to ${\mathcal S}(X)\cap {\mathcal T}(Y)$ is $X+P'$,
where $P'$ is a $1$-dimensional subspace of $P+Q$. 
Therefore, ${\mathcal S}^c(X)\cap {\mathcal T}^c(Y)$ contains $X+P$ only.
}\end{exmp}

Every maximal clique of $\Delta_k(V)$ can be considered as a point-line geometry
whose lines are defined in terms of  the adjacency relation. 
We use this concept to characterize  maximal stars of $\Delta_k(V)$ if  ${\mathbb F}$ is infinite.

\begin{prop}\label{prop-char-m-s}
The following assertions are fulfilled:
\begin{enumerate} 
\item[{\rm (1)}] In the case when ${\mathbb F}={\mathbb F}_q$, a maximal clique of $\Delta_k(V)$ is a maximal star if and only if 
it contains precisely $[n-k+1]_q$ vertices.
\item[{\rm (2)}] In the case when ${\mathbb F}$ is infinite, 
a maximal clique of $\Delta_k(V)$ is a maximal star if and only if this clique together with all lines it contains is a projective space. 
\end{enumerate}
\end{prop}

\begin{proof}
Every maximal star of $\Delta_k(V)$ is a star of $\Gamma_k(V)$ and every line of $\Delta_k(V)$ contained in a maximal star is a line of $\Gamma_k(V)$,
i.e. every maximal star of $\Delta_k(V)$ together with all lines it contains is an $(n-k)$-dimensional projective space over ${\mathbb F}$.

(1).
Suppose that ${\mathbb F}={\mathbb F}_q$. 
Then maximal and non-maximal  stars of $\Delta_k(V)$ contain precisely $[n-k+1]_q$ vertices and, respectively,  less than $[n-k+1]_q$ vertices.
By \cite[Lemma 2]{KP2}, a top of $\Delta_k(V)$ consists of $[k+1]_q-l$ vertices for a certain $l\in\{k+1,\dots,n\}$.
If $2k\le n$, then 
$$[n-k+1]_q\ge [k+1]_q,$$
i.e. every maximal clique of $\Delta_k(V)$ distinct from a maximal star contains less than $[n-k+1]_q$ vertices.
If $2k>n$, then
$$[k+1]_q-n>[n-k+1]_q$$
(see \cite[Lemma 4]{KP2} for the details) which means that 
every top of $\Delta_k(V)$ contains greater than $[n-k+1]_q$ vertices;
therefore, every maximal clique of $\Delta_k(V)$ consisting of $[n-k+1]_q$ vertices is a maximal star.

(2). Suppose that ${\mathbb F}$ is infinite.
Every non-maximal star ${\mathcal S}={\mathcal S}^c(X)$ together with all lines it contains 
is the point-line geometry obtained from the projective space related to ${\mathcal S}(X)$ by removing a finite number of hyperplanes $H_1,\dots, H_m$. 
In this projective space, we take the plane spanned by points $X_1,X_2,X_3$ belonging to ${\mathcal S}$. 
Let $\ell_1,\ell_2,\ell_3$ be the lines of the projective space joining distinct $X_i$ and $X_j$.
Since ${\mathbb F}$ is infinite, the plane contains two distinct lines $\ell,\ell'$
passing through a point on a certain hyperplane $H_i$ and intersecting each $\ell_j$ in a point belonging to ${\mathcal S}$. 
Consider the subspace ${\mathcal X}$ of the point-line geometry ${\mathcal S}$ spanned by $X_1,X_2,X_3$. 
It contains the lines $\ell_j\cap {\mathcal S}$, $j\in \{1,2,3\}$. Each of these lines intersects both $\ell\cap {\mathcal S}$ and  $\ell'\cap {\mathcal S}$
which implies that the latter two lines are contained in  ${\mathcal X}$.
Since $\ell\cap {\mathcal S}$ and  $\ell'\cap {\mathcal S}$ are non-intersecting, ${\mathcal X}$  is not a projective plane 
and, consequently, the point-line geometry ${\mathcal S}$ is not a projective space.

Similarly, every top ${\mathcal T}={\mathcal T}^c(Y)$ together with all lines it contains 
is the point-line geometry obtained from the projective space related to ${\mathcal T}(Y)$ by removing a finite number of points $Y_1,\dots, Y_m$.
Since ${\mathbb F}$ is infinite, this projective space contain a plane ${\mathcal Y}$ with precisely one $Y_i$. 
Then ${\mathcal Y}\setminus \{Y_i\}$ is a subspace of ${\mathcal T}$ which is not a projective plane, i.e.
the point-line geometry ${\mathcal T}$ is not a projective space.
\end{proof}

Tops and non-maximal stars of $\Delta_k(V)$ can be characterized via their  intersections with maximal stars. 

\begin{prop}\label{prop-char-t-s}
A maximal clique of $\Delta_k(V)$ distinct from a maximal star  is a top if and only if its intersection with every maximal star is  the empty set or a line. 
A maximal clique of $\Delta_k(V)$ distinct from a maximal star  is a non-maximal star if and only if there is a maximal star intersecting this clique precisely in one vertex.
\end{prop}

\begin{proof}
If the intersection of a top ${\mathcal T}^c(Y)$ and a maximal  star ${\mathcal S}(X)$ is non-empty, then the non-degenerate $(k-1)$-dimensional 
subspace  $X$ is contained in $Y$ which implies that this intersection is a line.

We need to show that for every non-maximal star ${\mathcal S}^c(X)$ there is a maximal star intersecting ${\mathcal S}^c(X)$ precisely in one vertex. 
We take any vector $x$ such that all coordinates of $x$ are  non-zero. 
Then $x\not\in X$ (since $X$ is degenerate)
and there is a non-degenerate $(k-1)$-dimensional subspace $X'$ containing $x$ and intersecting $X$
in a $(k-2)$-dimensional subspace. The maximal star ${\mathcal S}(X')$ is as required. 
\end{proof}

\section{Recovering of the Grassmann graph via stars or tops of the graph of non-degenerate subspaces}
Consider the simple graphs $\Delta_s$ and $\Delta_t$ whose vertices are stars and tops of $\Delta_k(V)$, respectively,
and whose edges are defined as follows: 
two distinct maximal cliques ${\mathcal X}_1$ and ${\mathcal X}_2$ (of the same type) are adjacent vertices of  the graph
$\Delta_x$, $x\in \{s,t\}$  if and only if  for each $i\in \{1,2\}$ every vertex of ${\mathcal X}_i$ is adjacent to at least one vertex of ${\mathcal X}_{3-i}$. 
This condition is satisfied, for example, if ${\mathcal X}_1$ and ${\mathcal X}_2$ have a non-empty intersection
(since these are distinct maximal cliques of the same type, this intersection is a single vertex).
Removing from $\Delta_x$ all edges connecting such pairs of cliques we obtain a simple graph which will be denoted by $\Delta'_x$.
So, $\Delta_x$ and $\Delta'_x$ have the same set of vertices and two distinct maximal cliques (of the corresponding type)
are adjacent vertices of $\Delta'_x$ if and only if they are adjacent in $\Delta_x$ and their intersection is empty.
Since $\Delta'_x$ is a subgraph of $\Delta_x$, every maximal clique of $\Delta'_x$ is a (not necessarily  maximal) clique of $\Delta_x$.

A set of stars  or a set of tops of $\Delta_k(V)$  is said to be {\it special} if it is a maximal clique in both $\Delta_s$ and $\Delta'_s$ or 
$\Delta_t$ and $\Delta'_t$, respectively.

Let us now consider the simple graphs $\Gamma_s$ and $\Gamma_t$ 
whose vertices are non-degenerate $k$-dimen\-sional subspaces (vertices of $\Delta_k(V)$)
and special sets of stars or special sets of tops, respectively. 
The edges of $\Gamma_x$, $x\in \{s,t\}$ are defined as follows:
\begin{enumerate} 
\item[$\bullet$] Two non-degenerate $k$-dimensional subspaces are adjacent vertices of $\Gamma_x$ if and only if they are adjacent.
\item[$\bullet$] A non-degenerate $k$-dimensional subspace and a special set of stars or tops are connected by an edge if and only if 
the subspace is a vertex in  a clique belonging to the special set.
\item[$\bullet$] Two distinct special sets of stars are adjacent vertices of $\Gamma_s$ if and only if there is a star belonging to each of these special sets.
\item[$\bullet$]
Two distinct special sets of tops ${\mathfrak T}_1$ and ${\mathfrak T}_2$ are adjacent vertices of $\Gamma_t$ if and only if
there is a star of $\Delta_k(V)$ which has a non-empty intersection with every top belonging to ${\mathfrak T}_i$, $i\in \{1,2\}$.
\end{enumerate}

\begin{theorem}\label{theorem-s-t}
Suppose that $|{\mathbb F}|>n-k$. Then the following assertions are fulfilled:
\begin{enumerate} 
\item[{\rm (1)}] 
If $k\ge 3$, then there is an isomorphism between the graphs $\Gamma_s$ and $\Gamma_k(V)$ whose restriction to $\Delta_k(V)$ is identity. 
\item[{\rm (2)}]
If $k\le n-3$, then there is an isomorphism between the graphs $\Gamma_t$ and $\Gamma_k(V)$ whose restriction to $\Delta_k(V)$ is identity. 
\end{enumerate}
\end{theorem}

Since the types of maximal cliques of $\Delta_k(V)$ can be determined in terms of adjacency, cardinalities and intersections
(Propositions \ref{prop-char-m-s} and \ref{prop-char-t-s}), Theorem \ref{theorem-s-t} provides two possibilities for recovering
of $\Gamma_k(V)$ from $\Delta_k(V)$ if $|{\mathbb F}|>n-k$ and $n\ge 5$. 
So,  Theorem \ref{theorem-1} follows from Theorem \ref{theorem-s-t} and Propositions \ref{prop-char-m-s} and \ref{prop-char-t-s} if $n\ge 5$. 

We assume  that $|{\mathbb F}|>n-k$.
In Section 5, we show that  
the graphs $\Delta_s$ and $\Delta_t$ are isomorphic to $\Gamma_{k-1}(V)$ and $\Delta_{k+1}(V)$, respectively.
In particular, $\Delta_s$ is a complete graph for $k=2$ and $\Delta_t$ is a complete graph for $k=n-2$. 
Therefore, there are no special sets of stars as well as special sets of tops if $n=4$.
In this case, the Grassmann graph will be recovered  via one class of maximal cliques of $\Delta'_t$ (Section 6).

\section{Proof of Theorem \ref{theorem-s-t}}
Throughout this section we suppose that $|{\mathbb F}|>n-k$. 
Recall that $1<k<n-1$ and, consequently, $n\ge 4$.

\subsection{Lemmas} 
To prove Theorem \ref{theorem-s-t} we will use the following lemmas.
\begin{lemma}\label{lemma-F}
For any non-degenerate $(k+1)$-dimensional subspace $Y$ and any $(k-1)$-dimensional subspace $X\subset Y$
there is a non-degenerate $k$-dimensional subspace $A$ satisfying $X\subset A\subset Y$. 
\end{lemma}

\begin{proof}
We take any $2$-dimensional subspace $L$ satisfying $L+X=Y$.
Let $I$ be the set of all $i\in \{1,\dots,n\}$ such that $X$ is contained in the coordinate hyperplane $C_i$.
If $i\in I$, then $C_i$ does not contain $L$
(otherwise $Y=L+X$ is contained in $C_i$ which contradicts the assumption that $Y$ is non-degenerate),
i.e. $C_i$ intersects $L$ in a $1$-dimensional subspace.
Note that $L$ contains precisely $|{\mathbb F}|+1$ distinct $1$-dimensional subspaces. 
Since
$$|I|\le n-\dim X=n-k+1<|{\mathbb F}|+1,$$
there is a $1$-dimensional subspace $P\subset L$ which is not contained in $C_i$ if $i\in I$.
The $k$-dimensional subspace $A=X+P$ is non-degenerate.
Indeed, if $A\subset C_i$, then $X\subset C_i$ which implies that $i\in I$
and, consequently, $P\not\subset C_i$, a contradiction.
\end{proof}

\begin{lemma}\label{lemma-top}
The set ${\mathcal T}^c(Y)$ is a top of $\Delta_k(V)$ for every non-degenerate $(k+1)$-dimensional subspace $Y$.
\end{lemma}

\begin{proof}
This follows from  the statement (2) of Proposition \ref{prop-top} if ${\mathbb F}$ is infinite.
Suppose that ${\mathbb F}={\mathbb F}_q$. 
By the statement (1) of Proposition \ref{prop-top},
it is sufficient to show that the condition $|{\mathbb F}|=q>n-k$ guarantees that 
$$[k+1]_q-(q+1)>n.$$
If $k=2$, then
$$q^2>(n-2)^2\ge n$$
(since $n\ge 4$) which implies the required inequality. 
Instead of showing the same for $k\ge 3$ we use some simple geometric reasonings in this case.

By Lemma \ref{lemma-F}, ${\mathcal T}^c(Y)$ is non-empty.
If ${\mathcal T}^c(Y)$ is not a maximal clique of $\Delta_k(V)$, then it is contained in a maximal clique of $\Delta_k(V)$,
i.e. a star or a top. In the second case, ${\mathcal T}^c(Y)$ contains precisely one vertex. 
Therefore, there is a $(k-1)$-dimensional subspace $X$ contained in all $k$-dimensional subspaces belonging to ${\mathcal T}^c(Y)$. 
We take any $(k-1)$-dimensional subspace $X'\subset Y$ such that $$\dim (X\cap X')=k-3$$ 
(recall that $k\ge 3$ by our assumption). 
Lemma \ref{lemma-F} implies the existence of $A\in {\mathcal T}^c(Y)$ containing $X'$. Then $X$ and $X'$ both are contained in
the $k$-dimensional subspace $A$ which means that 
$$\dim (X\cap X')\ge k-2,$$ 
a contradiction.
\end{proof}

\begin{lemma}\label{lemma-S}
If ${\mathcal S}_i={\mathcal S}^c(X_i)$, $i\in \{1,2\}$ are distinct stars of $\Delta_k(V)$, then  the following conditions are equivalent:
\begin{enumerate}
\item[{\rm (1)}] $X_1$ and $X_2$ are adjacent, 
\item[{\rm (2)}] for every $A\in {\mathcal S}_i$, $i\in \{1,2\}$ there is $B\in {\mathcal S}_{3-i}$ adjacent to $A$. 
\end{enumerate}
\end{lemma}

\begin{proof}
$(1)\Rightarrow(2)$.
If $A\in {\mathcal S}_i\cap {\mathcal S}_{3-i}$, then (2) is obvious.
Let $A\in {\mathcal S}_i\setminus {\mathcal S}_{3-i}$. 
Since $X_1$ and $X_2$ are adjacent,
the  subspace $A+X_{3-i}$ is $(k+1)$-dimensional. 
It also is non-degenerate (as containing the non-degenerate subspace $A$). By Lemma \ref{lemma-F}, 
there is a non-degenerate $k$-dimensional subspace $B\subset A+X_{3-i}$ containing $X_{3-i}$. 
Then $B\in {\mathcal S}_{3-i}$ is adjacent to $A$.

$(2)\Rightarrow(1)$. 
This is trivial if $k=2$ (any two distinct $1$-dimensional subspaces are adjacent).
Let $k\ge 3$. Show that (2) fails if  $X_1$ and $X_2$ are not adjacent. 
Suppose that 
$$\dim(X_1\cap X_2)\le k-3$$
and (2) is satisfied. 
Then there are adjacent $A\in {\mathcal S}_1$ and $B\in {\mathcal S}_2$ which means that $X_1$ and $X_2$ are 
contained in the $(k+1)$-dimensional subspace $A+B$. 
Therefore, 
$$\dim(X_1\cap X_2)=k-3$$
which implies that
$$\dim(X_1+X_2)=k+1.$$
Since ${\mathcal S}_i$ is a star of $\Delta_k(V)$, there is $A\in {\mathcal S}_i$ which is not contained in 
the $(k+1)$-dimensional subspace $X_1+X_2$. 
By (2), $A$ is adjacent to a certain $B\in {\mathcal S}_{3-i}$.
The $(k+1)$-dimensional subspace $A+B$ contains $X_1,X_2$ and, 
consequently, coincides with $X_1+X_2$. This contradicts the fact that $A$ is not contained in $X_1+X_2$. 
\end{proof}

\begin{lemma}\label{lemma-T}
If ${\mathcal T}_i={\mathcal T}^c(Y_i)$, $i\in \{1,2\}$ are distinct tops of $\Delta_k(V)$, then the following conditions are equivalent:
\begin{enumerate}
\item[{\rm (1)}] $Y_1$ and $Y_2$ are adjacent,
\item[{\rm (2)}] for every $A\in {\mathcal T}_i$, $i\in \{1,2\}$ there is $B\in {\mathcal T}_{3-i}$ adjacent to $A$. 
\end{enumerate}
\end{lemma}

\begin{proof}
$(1)\Rightarrow(2)$.
If $A\in {\mathcal T}_i\cap {\mathcal T}_{3-i}$, then (2) is obvious.
Let $A\in {\mathcal T}_i\setminus {\mathcal T}_{3-i}$. 
Then $A$ and $Y_1\cap Y_2$ are distinct $k$-dimensional subspaces of $Y_i$ and
their intersection is a $(k-1)$-dimensional subspace $X'$.
Since $X'\subset Y_{3-i}$, Lemma \ref{lemma-F} implies the existence of  a non-degenerate $k$-dimensional subspace $B\subset Y_{3-i}$
containing $X'$.
So, $B\in {\mathcal T}_{3-i}$ is adjacent to $A$.

$(2)\Rightarrow(1)$. 
This is trivial if $k=n-2$ (any two distinct $(n-1)$-dimensional subspaces are adjacent). 
Let $k\le n-3$. 
Show that (2) fails if  $Y_1$ and $Y_2$ are not adjacent. 
Suppose that  
$$\dim(Y_1\cap Y_2)<k$$ and (2) is satisfied. 
Then there are adjacent $A\in {\mathcal T}_1$ and $B\in {\mathcal T}_2$.
Their intersection  is a $(k-1)$-dimensional subspace contained in $Y_1\cap Y_2$ which implies that
$$\dim(Y_1\cap Y_2)=k-1.$$
Since ${\mathcal T}_i$ is a top of $\Delta_k(V)$, there is $A\in {\mathcal T}_i$ which does not contain $Y_1\cap Y_2$.
Then $A$ intersects $Y_{3-i}$ in a subspace of dimension less than $k-1$ and, consequently, 
there is no $B\in {\mathcal T}_{3-i}$ adjacent to $A$,  a contradiction. 
\end{proof}

\subsection{Recovering via stars}
Since $1<k<n-1$, we have $|{\mathbb F}|>n-k\ge 2$ and Proposition \ref{prop-star} shows that 
${\mathcal S}^c(X)$ is a star of $\Delta_k(V)$ for every $(k-1)$-dimensional subspace $X$.

Recall that the vertex set of the graphs $\Delta_s$ and $\Delta'_s$ is the set of stars of  $\Delta_k(V)$. 
Two distinct stars ${\mathcal S}_1$ and ${\mathcal S}_2$ are adjacent vertices of $\Delta_s$ if and only if 
for each $i\in \{1,2\}$ every vertex of ${\mathcal S}_i$ is adjacent to at least one vertex of ${\mathcal S}_{3-i}$. By Lemma \ref{lemma-S},
this is equivalent to the fact that the corresponding $(k-1)$-dimensional subspaces are adjacent.
So, the graph $\Delta_s$ is isomorphic to $\Gamma_{k-1}(V)$; in particular, $\Delta_s$  is a complete graph if $k=2$. 

Two distinct stars of  $\Delta_k(V)$ are adjacent vertices of  $\Delta'_s$ 
if and only if they are adjacent in $\Delta_s$ and their intersection is empty or, equivalently,
the sum of the corresponding $(k-1)$-dimensional subspaces is a degenerate $k$-dimensional subspace. 
The latter implies that these stars are non-maximal. 
Therefore, every maximal star of $\Delta_k(V)$ is an isolated vertex of $\Delta'_s$
and cliques of $\Delta'_s$ containing more than one vertex are formed by non-maximal stars.

For a degenerate $k$-dimensional subspace $X$ denote by ${\mathfrak S}_X$
the set of all stars of $\Delta_k(V)$ such that the corresponding $(k-1)$-dimensional subspaces are 
contained in $X$, in other words, form the top of $\Gamma_{k-1}(V)$ defined by $X$. 
This is a maximal clique in both $\Delta_s$ and $\Delta'_s$, i.e. a special set of stars, if $k\ge 3$.
For $k=2$ this clique is not maximal in $\Delta_s$.

\begin{lemma}\label{lemma-sS}
If $k\ge 3$ and ${\mathfrak S}$ is a special set of stars, then ${\mathfrak S}={\mathfrak S}_X$ for a certain degenerate $k$-dimensional subspace $X$.
\end{lemma}

\begin{proof}
Since ${\mathfrak S}$ is a maximal clique of $\Delta_s$,
the $(k-1)$-dimensional subspaces corresponding to stars from ${\mathfrak S}$ form a maximal clique ${\mathcal X}$ in $\Gamma_{k-1}(V)$.
Then ${\mathfrak S}$ is a clique of $\Delta'_s$ containing more than one vertex and, consequently,
it is formed by non-maximal stars which means that ${\mathcal X}$ consists of degenerate subspaces. 
Since $k\ge 3$, ${\mathcal X}$ is a star or a top of $\Gamma_{k-1}(V)$.

Suppose that ${\mathcal X}$ is the top of $\Gamma_{k-1}(V)$ defined by a $k$-dimensional subspace $X$.
For any distinct $A,B\in {\mathcal X}$ the corresponding stars of $\Delta_k(V)$ are adjacent in $\Delta'_s$
which implies that $A+B=X$ is degenerate.
Therefore, ${\mathfrak S}={\mathfrak S}_X$. 

Suppose that ${\mathcal X}$ is the star of $\Gamma_{k-1}(V)$ defined by a $(k-2)$-dimensional subspace $Z$.
Then $Z$ is degenerate (as contained in degenerate subspaces). 
If $P$ is a non-degenerate $1$-dimensional subspace, then $P\not\subset Z$ and the $k$-dimensional subspace $Z+P$ is non-degenerate. 
The maximal star ${\mathcal S}(Z+P)$ does not belong to ${\mathfrak S}$, but it is adjacent to all stars from ${\mathfrak S}$ in $\Delta_s$.
So, the clique ${\mathfrak S}$ is not maximal in $\Delta_s$.
\end{proof}

As above, we suppose that $k\ge 3$.
Let $f$ be the bijection between the vertex sets of $\Gamma_s$ and $\Gamma_k(V)$ 
sending every non-degenerate $k$-dimensional subspace to itself and every special set  ${\mathfrak S}_X$ to $X$.
It is clear that the restriction of $f$ to the set of all non-degenerate $k$-dimensional subspaces is adjacency preserving in both directions.
A non-degenerate $k$-dimensional subspace $A$ and a special set ${\mathfrak S}_B$ are adjacent  in $\Gamma_s$ if and only if $A$ belongs to a certain star from ${\mathfrak S}_B$.
Two distinct special sets ${\mathfrak S}_A$ and ${\mathfrak S}_B$ are adjacent in $\Gamma_s$ if and only if there is a
star belonging to both ${\mathfrak S}_A$ and ${\mathfrak S}_B$.
Each of these conditions is equivalent to the existence of a $(k-1)$-dimensional subspace contained in both $A,B$
which holds if and only if $A$ and $B$ are adjacent. 
So, $f$ is a graph isomorphism.

\subsection{Recovering via tops}
By Lemma \ref{lemma-top}, ${\mathcal T}^c(Y)$ is a top of $\Delta_k(V)$ for every non-degenerate $(k+1)$-dimensional subspace $Y$.

The vertex set of the graphs $\Delta_t$ and $\Delta'_t$ is the set of tops of  $\Delta_k(V)$.
Two distinct tops ${\mathcal T}_1$ and ${\mathcal T}_2$ are adjacent vertices of $\Delta_t$ if and only if 
for each $i\in \{1,2\}$ every vertex of ${\mathcal T}_i$ is adjacent to at least one vertex of ${\mathcal T}_{3-i}$. By Lemma \ref{lemma-T},
this is equivalent to the fact that the corresponding $(k+1)$-dimensional subspaces are adjacent. 
Therefore, the graph $\Delta_t$ is isomorphic to $\Delta_{k+1}(V)$; in particular, $\Delta_t$  is a complete graph if $k=n-2$. 

Two distinct tops of  $\Delta_k(V)$ are adjacent vertices of $\Delta'_t$ 
if and only if they are adjacent in $\Delta_t$ and their intersection is empty or, equivalently, 
the intersection of the corresponding $(k+1)$-dimensional subspaces is a degenerate $k$-dimensional subspace. 

For a degenerate $k$-dimensional subspace $X$ denote by ${\mathfrak T}_X$ 
the set of all tops of $\Delta_k(V)$ such that the corresponding non-degenerate $(k+1)$-dimensional subspaces contain $X$.
If $k\le n-3$, then  $|{\mathbb F}| >n-k\ge 3$ and
Proposition \ref{prop-star} shows that the non-degenerate $(k+1)$-dimensional subspaces containing $X$ form a star of $\Delta_{k+1}(V)$
which implies that ${\mathfrak T}_X$ is a maximal clique in both $\Delta_t$ and $\Delta'_t$, i.e. a special set of tops.
For $k=n-2$ this clique is not maximal in $\Delta_t$.

\begin{lemma}\label{lemma-sT}
If $k\le n-3$ and  ${\mathfrak T}$ is a special set of tops,  then ${\mathfrak T}={\mathfrak T}_X$ for a certain degenerate $k$-dimensional subspace $X$.
\end{lemma}

\begin{proof}
Since ${\mathfrak T}$ is a maximal clique of $\Delta_t$,
the non-degenerate $(k+1)$-dimensional subspaces corresponding to tops from ${\mathfrak T}$
form a maximal clique ${\mathcal X}$ in $\Delta_{k+1}(V)$. 
This is a star or a top of $\Delta_{k+1}(V)$ as $k\le n-3$.

Suppose that ${\mathcal X}$ is the star of $\Delta_{k+1}(V)$ defined by a $k$-dimensional subspace $X$. 
For any distinct $A,B\in {\mathcal X}$ the corresponding tops of $\Delta_k(V)$ are adjacent in $\Delta'_t$
which implies that $A\cap B=X$ is degenerate.
Therefore, ${\mathfrak T}={\mathfrak T}_X$. 

Suppose that ${\mathcal X}$ is the top of $\Delta_{k+1}(V)$ defined by a non-degenerate $(k+2)$-dimen\-sional subspace $Z$. 
Let ${\mathcal T}^c(Y)$ be a top belonging to ${\mathfrak T}$ and let $A\in {\mathcal T}^c(Y)$.
We take any $(k+1)$-dimensional subspace  $Y'$ satisfying $A\subset Y'\subset Z$ and distinct from $Y$. 
Then $Y'$ is non-degenerate (as containing the non-degenerate subspace $A$).
Since $Y\cap Y'=A$, the tops ${\mathcal T}^c(Y)$ and ${\mathcal T}^c(Y')$ are not adjacent in $\Delta'_t$
and, consequently, ${\mathcal T}^c(Y')$ does not belong to ${\mathfrak T}$.
On the other hand, $Y'\subset Z$ and, consequently,  ${\mathcal T}^c(Y')$ is adjacent to all tops from ${\mathcal T}$ in $\Delta_t$. 
So, the clique ${\mathcal T}$ is not maximal in $\Delta_t$.
\end{proof}

As above, we suppose that $k\le n-3$. 
Let $f$ be the bijection between the vertex sets of $\Gamma_t$ and $\Gamma_k(V)$ 
sending every non-degenerate $k$-dimensional subspace to itself and every special set  ${\mathfrak T}_X$ to $X$.
The restriction of $f$ to the set of all non-degenerate $k$-dimensional subspaces is adjacency preserving in both directions.
A non-degenerate $k$-dimensional subspace $A$ and a special set ${\mathfrak T}_B$ are adjacent  in $\Gamma_t$ if and only if $A$ belongs to a certain top 
from ${\mathfrak T}_B$. This is equivalent to the existence of a $(k+1)$-dimensional subspace containing both $A,B$ 
which holds if and only if $A$ and $B$ are adjacent.

We need to show that two distinct special sets ${\mathfrak T}_A$ and ${\mathfrak T}_B$ are adjacent in $\Gamma_t$ if and only if 
$A$ and $B$ are adjacent. 

If ${\mathfrak T}_A$ and ${\mathfrak T}_B$ are adjacent in $\Gamma_t$, then
there is a star ${\mathcal S}^c(Z)$ of $\Delta_k(V)$ intersecting  all tops from ${\mathfrak T}_A$ and ${\mathfrak T}_B$. 
So, $Z$ is contained in every non-degenerate $(k+1)$-dimensional subspace containing $A$. 
We have not less than two such subspaces which means that $Z\subset A$. Similarly, we obtain that $Z\subset B$. 
Since $Z$ is $(k-1)$-dimensional, $A$ and $B$ are adjacent. 

Suppose that $A$ and $B$ are adjacent.
By Lemma \ref{lemma-F}, every non-degenerate $(k+1)$-dimensional subspace containing $A$ 
contains a non-degenerate $k$-dimensional subspace containing $A\cap B$.
Therefore, the star ${\mathcal S}^c(A\cap B)$ intersects  every top from ${\mathfrak T}_A$.
Similarly, this star  intersects every top from ${\mathfrak T}_B$.
We obtain that ${\mathfrak T}_A$ and ${\mathfrak T}_B$ are adjacent in $\Gamma_t$. 

\section{Proof of Theorem \ref{theorem-1} for $n=4$}
Suppose that $n=4$ and $k=2$. Then $|{\mathbb F}|>n-k=2$.
The graphs $\Delta_s$ and $\Delta_t$ are complete and 
we determine maximal cliques of $\Delta'_s$ and $\Delta'_t$.

Two distinct stars of $\Delta_k(V)$ are adjacent vertices of $\Delta'_s$ if and only if the sum of the corresponding $1$-dimensional subspaces is degenerate.
This shows that  there precisely the following two types of maximal cliques of  $\Delta'_s$:
\begin{enumerate}
\item[$\bullet$] an isolated vertex which is a maximal star,
\item[$\bullet$] the set of all stars such that the corresponding $1$-dimensional subspaces are contained in a certain $C_i$.
\end{enumerate}
So, it will be difficult to  use the graph $\Delta'_s$ to recovering of $\Gamma_k(V)$.

Consider the graph $\Delta'_t$.
As in the previous section, for every degenerate $2$-dimensional subspace $X$ we denote by ${\mathfrak T}_X$
the set of all tops of $\Delta_k(V)$ such that the corresponding $3$-dimensional subspaces contain $X$. 
This is a clique of $\Delta'_t$. We show that such cliques are maximal and characterize them in terms of cardinalities and intersections.

Let ${\mathcal P}^*$ be the projective space dual to the projective space associated to $V$. 
Recall that points and lines of ${\mathcal P}^*$ correspond to $3$-dimensional and, respectively, $2$-dimensional subspaces of $V$. 
The points of ${\mathcal P}^*$ corresponding to the coordinate hyperplanes $C_1,C_2,C_3,C_4$ will be denoted by $c_1,c_2,c_3,c_4$ and called {\it degenerate}. 
Note that ${\mathcal P}^*$ is spanned by all degenerate points and any three mutually distinct degenerate points are non-collinear.

The remaining points of  ${\mathcal P}^*$  are said to be {\it non-degenerate}. Such points are identified with tops of $\Delta_k(V)$, i.e. vertices of $\Delta'_t$.
A line of ${\mathcal P}^*$ contains a degenerate point if and only if the corresponding $2$-dimensional subspace of $V$ is degenerate.
Therefore, two distinct non-degenerate points of ${\mathcal P}^*$  are adjacent vertices of $\Delta'_t$ if and only if the line joining them contains a degenerate point.

\begin{prop}\label{prop4-1}
For every degenerate $2$-dimensional subspace $X$ the clique ${\mathfrak T}_X$ is maximal. 
\end{prop}

\begin{proof}
Since  $|{\mathbb F}|\ge 3$, every line of ${\mathcal P}^*$ contains at least four points. 
Let $\ell$ be the line of ${\mathcal P}^*$ corresponding to $X$. 

Suppose that $\ell$ contains precisely one degenerate point. Then there are at least three non-degenerate points on $\ell$. 
If a point $p\not\in \ell$ is adjacent to these three points in $\Delta'_t$, then the lines joining these points and $p$ contain
the remaining three degenerate points. 
This means that all degenerate points are contained in the plane spanned by $\ell$ and $p$ 
which  contradicts the fact that ${\mathcal P}^*$ is spanned by all degenerate points.

If $\ell$ contains two degenerate points, then there are at least two non-degenerate points on $\ell$.
If a point $p\not\in \ell$ is adjacent to these two points in $\Delta'_t$, then the lines joining these points and $p$ contain
the remaining two degenerate points. As above, we get a contradiction. 
\end{proof}

Maximal cliques of $\Delta'_t$ defined by degenerate $2$-dimensional subspaces of $V$ (lines of ${\mathcal P}^*$
passing through degenerate points) are said to be {\it linear}. 
Show that there are {\it non-linear} maximal cliques of $\Delta'_t$, i.e. maximal cliques which are not contained in lines of ${\mathcal P}^*$.

Consider the plane spanned by $c_1,c_2,c_3$ and a line $\ell$ in this plane which  passes through $c_1$ and does not contain $c_2$ and $c_3$.
We take distinct points $p_2,p_3\in \ell$ which are different from $c_1$ and do not belong to the line joining $c_2$ and $c_3$
(such points exist, since $|{\mathbb F}|\ge 3$ and $\ell$ contains at least four points).
If $p_1$ is  the intersecting point of the  lines joining $c_2$ and $c_3$ with $p_2$ and $p_3$ (respectively),
then $p_1,p_2,p_3$ are non-collinear in ${\mathcal P}^*$ and mutually adjacent in $\Delta'_t$.
A maximal clique of $\Delta'_t$ containing these points is as required.

\begin{lemma}\label{lemma4-1}
For every non-linear maximal clique of $\Delta'_t$ the following assertions are fulfilled:
\begin{enumerate}
\item[{\rm (1)}] it is contained in one of the planes spanned by three degenerate points;
\item[{\rm(2)}] it contains at most four vertices;
\item[{\rm (3)}] it does not intersect the lines joining pairs of degenerate points.
\end{enumerate}
\end{lemma}

\begin{proof}
Let ${\mathcal X}$ be a maximal clique of $\Delta'_t$ containing three non-collinear points $p_1,p_2,p_3$.
Since the plane spanned by $p_1,p_2,p_3$ contains at most three degenerate points,
there is precisely one degenerate point on each line joining two distinct $p_i$.
This guarantees that the lines joining pairs of degenerate points do not contain $p_1,p_2,p_3$. 

So, $p_1,p_2,p_3$ belong to the plane spanned by three degenerate points.
Without loss of generality we assume that the degenerate points are $c_1,c_2,c_3$.
If $p\in {\mathcal X}$ is not on this plane, then each of the lines joining $p$ with $p_1,p_2,p_3$ passes through $c_4$ which is impossible.
So, ${\mathcal X}$ is contained in the plane spanned by $c_1,c_2,c_3$.  

Suppose that $p\in {\mathcal X}$ is distinct from $p_1,p_2,p_3$.
The arguments used to prove Proposition \ref{prop4-1} show that any three mutually distinct vertices of ${\mathcal X}$ are non-collinear points of ${\mathcal P}^*$.
Then the lines joining $p$ with $p_1,p_2,p_3$ are mutually distinct and each of these lines contains precisely one of $c_1,c_2,c_3$. 
If $q\in {\mathcal X}$ is distinct from $p_1,p_2,p_3,p$, then the line joining $p$ and $q$ contains one of $c_1,c_2,c_3$
and, consequently, this lines  contains one of $p_1,p_2,p_3$ which is impossible. 
Therefore, ${\mathcal X}=\{p_1,p_2,p_3,p\}$.

It was shown above that the lines joining pairs of degenerate points do not contain $p_1,p_2,p_3$.
Applying the same arguments to the non-collinear points $p_1,p_2,p$ we establish that 
these lines do not contain $p$. 
\end{proof}

\begin{prop}\label{prop4-2}
Suppose that $|{\mathbb F}|\ge 6$. Then a maximal clique of $\Delta'_t$ is linear if and only if it contains at least five vertices.
\end{prop}

\begin{proof}
A line of ${\mathcal P}^*$ contains at least $|{\mathbb F}|-1$ non-degenerate points.
Therefore, every linear maximal clique contains at least five vertices.
By Lemma \ref{lemma4-1}, non-linear maximal cliques contain at most four vertices. 
\end{proof}

For a line of ${\mathcal P}^*$ passing through a degenerate point one of the following possibilities is realized:
\begin{enumerate}
\item[(1)] it is not contained in the planes spanned by triples of degenerate points;
\item[(2)] it passes through two distinct degenerate points;
\item[(3)] it is contained in one of the planes spanned by three degenerate points and (2) is not satisfied.
\end{enumerate}
A linear maximal clique of $\Delta'_t$  is said to be {\it of type} $m\in \{1,2,3\}$
if the corresponding line of ${\mathcal P}^*$ satisfies $(m)$.

\begin{prop}\label{prop4-3}
If ${\mathbb F}={\mathbb F}_q$, then the following assertions are fulfilled: 
\begin{enumerate}
\item[{\rm(1)}] A maximal clique of $\Delta'_t$ is a linear maximal clique of type $1$ or $2$
if and only if it intersects any other maximal clique of $\Delta'_t$ in at most one vertex.
\item[{\rm (2)}] 
A linear maximal clique of type $1$ or $2$ is a linear maximal clique of type $2$  if and only if it consists of $q-1$ vertices.
\item[{\rm (3)}] 
A maximal clique of $\Delta'_t$ which is not a linear maximal clique of type $1$ or $2$
is a linear maximal clique of type $3$ if and only if 
it intersects a linear maximal clique of type $2$.
\end{enumerate}
\end{prop}

\begin{proof}
(1). The intersection of two distinct linear maximal cliques of $\Delta'_t$ contains at most one vertex. 

A linear maximal clique of type $1$ corresponds to a line of ${\mathcal P}^*$ which is not contained in the planes spanned by triples of degenerate points.
By the statement (1) of Lemma  \ref{lemma4-1}, such a clique intersects every non-linear maximal clique of $\Delta'_t$  in at most one vertex. 
So, the intersection of every linear maximal clique of type $1$ with any other maximal clique of $\Delta'_t$ contains  at most one vertex. 

The same holds for linear maximal cliques of type $2$, since they do not intersect non-linear maximal cliques  by the statement (3) of Lemma \ref{lemma4-1}.

For every linear maximal clique of type $3$ there is a non-linear  maximal clique of $\Delta'_t$ intersecting this clique precisely in two vertices 
(see the construction of a non-linear maximal clique  before Lemma \ref{lemma4-1}).
Similarly, for every non-linear maximal clique of $\Delta'_t$ there is a linear maximal clique of type $3$ intersecting this clique precisely in two vertices. 

(2). This is obvious. 

(3). A maximal clique of $\Delta'_t$ which is not a linear maximal clique of type $1$ or $2$
is a linear maximal clique of type $3$ or a non-linear maximal clique. 
The statement (3) of Lemma \ref{lemma4-1} gives the claim.
\end{proof}

Propositions \ref{prop4-2} and \ref{prop4-3} characterize linear maximal cliques of $\Delta'_t$ in terms of cardinalities and intersections. 
Consider the simple graph $\Gamma'_t$ whose vertices are non-degenerate $2$-dimensional subspaces and 
linear maximal cliques of $\Delta'_t$ and whose edges are defined as for the graph $\Gamma_t$ in Section 4.
The arguments from Subsection 5.3 show that the map sending every non-degenerate $2$-dimensional subspace to itself and every ${\mathfrak T}_X$ to $X$
is an isomorphism between $\Gamma'_t$ and $\Gamma_k(V)$.

\section{Final remarks}
In this section, we explain why our methods do not work for the general case.
The last two remarks describe relations to results obtained in \cite{KP2} and \cite{PZ}.

The following modification of Example \ref{exmp1}  shows that Lemma \ref{lemma-F}  fails if $|{\mathbb F}|\le n-k$.

\begin{exmp}\label{exmp2}{\rm
Suppose that ${\mathbb F}={\mathbb F}_q$ and $q=n-k$. As in  Example \ref{exmp1}, we take
$$X=C_1\cap\dots\cap C_{n-k+1}.$$ 
Let $P$ be the $1$-dimensional subspace containing the vector $(0,1,\dots,1)$
and let $Q$ be any $1$-dimensional subspace containing a vector 
$$(1,t_1,\dots,t_{n-k},\dots),$$
where $t_1,\dots, t_{n-k}$ are the elements of ${\mathbb F}_q$.
Then $P$ is contained in $C_1$. Every $1$-dimensional subspace of $P+Q$ distinct from $P$ contains the vector 
$$(0,t_i,\dots,t_i)-(1,t_1,\dots,t_{n-k},\dots)$$
for a certain $i\in \{1,\dots,n-k\}$ and, consequently, this subspace is contained in $C_j$ for a certain $j\in\{2,\dots,n-k+1\}$.
The $(k+1)$-dimensional subspace $$Y=X+P+Q$$ is non-degenerate (since $P+Q$ is non-degenerate).
Every $k$-dimensional subspace of $Y$ containing $X$ is $X+P'$, where $P'$ is a $1$-dimensional subspace of $P+Q$. 
It was established above that $P'$ is contained in $C_j$ for  a certain $j\in\{1,\dots,n-k+1\}$. Then $X+P'\subset C_j$.
So, all $k$-dimensional subspaces of $Y$ containing $X$ are degenerate.
}\end{exmp}

If $|{\mathbb F}|\le n-k$, then the implication $(1)\Rightarrow (2)$ from Lemma \ref{lemma-S} fails. 

\begin{exmp}\label{exmp3}{\rm
Suppose that ${\mathbb F}={\mathbb F}_q$ and $q=n-k\ge 3$. 
Let $X,P,Q,Y$ be as in Example \ref{exmp2}. 
Consider the $(k-1)$-dimensional subspaces 
$$X_1=X\;\mbox{ and }\;X_2=X'+P,$$ 
where $X'$ is a $(k-2)$-dimensional subspace of $X$. 
These subspaces are adjacent.
The $k$-dimensional subspace $$A_2=X_2+Q=X'+P+Q$$
is non-degenerate and belongs to the star ${\mathcal S}^c(X_2)$. 
Observe that $$X_1+A_2=X+A_2=X+P+Q=Y.$$
If $A_1\in {\mathcal S}^c(X_1)$ is adjacent to $A_2$, then $A_1+A_2$ is a $(k+1)$-dimensional subspace containing $X_1$ and $A_2$.
This means that  $A_1+A_2=Y$ and, consequently, $A_1\subset Y$. 
Then $A_1$ is a non-degenerate $k$-dimensional subspace of $Y$ containing $X_1=X$ which is impossible by Example \ref{exmp2}.
So,  there is no $A_1\in {\mathcal S}^c(X_1)$ adjacent to $A_2$. 
}\end{exmp}

\begin{rem}{\rm
If $|{\mathbb F}|\le n-k$, then
for stars ${\mathcal S}_1$ and ${\mathcal S}_2$  of  $\Delta_k(V)$ corresponding to $(k-1)$-dimensional subspaces $X_1$ and $X_2$ (respectively)
we are not able  to distinguish the following two possibilities:
\begin{enumerate}
\item[$\bullet$] $X_1,X_2$ are degenerate and $\dim(X_1\cap X_2)=k-2$;
\item[$\bullet$] $X_1,X_2$ are non-degenerate and $\dim(X_1\cap X_2)=k-3$.
\end{enumerate}
For each of these cases we have ${\mathcal S}_1\cap {\mathcal S}_2=\emptyset$ and 
there are adjacent $A_1\in{\mathcal S}_1$ and $A_2\in {\mathcal S}_2$.
In the first case, we take $A_i=X_i+P$, $i\in\{1,2\}$ for any non-degenerate $1$-dimensional subspace $P\not\subset X_1+X_2$.
In the second, $X_1+X_2$ is a non-degenerate $(k+1)$-dimensional subspace  and any $A_i\in{\mathcal S}_i$, $i\in\{1,2\}$ contained in this subspace are 
as required. 
}\end{rem}

\begin{rem}{\rm
By Theorem \ref{theorem-1}, every automorphism of $\Delta_k(V)$ is an automorphism of $\Gamma_k(V)$ if  $|{\mathbb F}|>n-k$. 
Automorphisms of Grassmann graphs are known (see, for example, \cite{Chow}).
Every automorphism of $\Gamma_k(V)$ is induced by a semilinear  automorphism of $V$ or a semilinear isomorphism of $V$ to $V^*$
(the second possibility is realized only when $n=2k$). 
Using these facts, we can show that every automorphism of $\Delta_k(V)$ is induced by a monomial semilinear automorphism of $V$
under the assumption that $|{\mathbb F}|>n-k$.
In \cite{KP2}, the same is obtained for the graph of non-degenerate linear $[n,k]_q$ codes without any assumptions.
Maximal cliques are used to reduce the general case to the case when $k=2$;
in the latter case,  some arguments in spirit of the Fundamental Theorem of Projective Geometry are applied. 
See \cite{P1,P2} for more strong results.
}\end{rem}

\begin{rem}{\rm
Some results on recovering point-line geometries from their subsets can be found in \cite{PZ}. 
Let ${\mathcal W}$ be a subset of a point-line geometry. 
The {\it index} of ${\mathcal W}$ is defined as a number $\lambda$ such that 
there is a line containing precisely $\lambda$ points belonging to ${\mathcal W}$ 
and every line contains at most $\lambda$ points belonging to ${\mathcal W}$ or it is contained in ${\mathcal W}$. 
One of the main conditions used to recover the point-line geometry from the complement of ${\mathcal W}$ is the following:
every line contains at least $2\lambda +2$ points, see \cite[Theorems 3.1 and 5.1]{PZ}.
If ${\mathbb F}={\mathbb F}_q$ and $q=n-k+1$, then  Example \ref{exmp1} shows that 
the index of the set of degenerate $k$-dimensional subspaces is $q$ and the above condition is not satisfied (every line contains precisely $q+1$ points).
So, results of \cite{PZ} are not applicable.
}\end{rem}

\end{document}